\documentclass[12pt]{amsart}

\usepackage{amscd,latexsym,amsthm,amsfonts,amssymb,amsmath,amsxtra,eucal}
\usepackage{graphicx}
\usepackage{genyoungtabtikz}
\usepackage{pifont}
\usepackage{multirow}
\usepackage{threeparttable}
\newcommand{\hobox}[3]{\draw (0+#1,0-#2) rectangle (1+#1,-1-#2)++(-0.5,+0.5) node {$ #3$};}

\newcommand{\domscale}{0.5}
\usepackage[all]{xy}

\usepackage{ytableau}

\usepackage[pdftex,bookmarks,colorlinks,pdfmenubar]{hyperref}

\usepackage{verbatim}
\usepackage{mathabx}
\usepackage{mathrsfs}
\usepackage{bbm}

\renewcommand\theequation{\thesection.\arabic{equation}}

\newcommand{\BC}{{\mathbb {C}}}

\newcommand{\BN}{{\mathbb {N}}}

\newcommand{\BR}{{\mathbb {R}}}

\newcommand{\BZ}{{\mathbb {Z}}}

\newcommand{\CC}{{\mathcal {C}}}
\renewcommand{\CD}{{\mathcal {D}}}

\newcommand{\CK}{{\mathcal {K}}}

\newcommand{\CO}{{\mathcal {O}}}
\newcommand{\CP}{{\mathcal {P}}}

\newcommand{\CU}{{\mathcal {U}}}

\newcommand{\CY}{{\mathcal {Y}}}

\newcommand{\FS}{{\mathfrak {S}}}

\newcommand{\Fb}{{\mathfrak {b}}}

\newcommand{\Fg}{{\mathfrak {g}}}
\newcommand{\Fh}{{\mathfrak {h}}}

\newcommand{\Fl}{{\mathfrak {l}}}

\newcommand{\RU}{{\mathrm {U}}}

\newcommand{\GL}{{\mathrm{GL}}}

\newcommand{\Hom}{{\mathrm{Hom}}}

\newcommand{\Ind}{{\mathrm{Ind}}}

\newcommand{\Lie}{{\mathrm{Lie}}}

\newcommand{\Sym}{{\mathrm{Sym}}}

\newcommand{\wh}{\widehat}

\newcommand{\Dim}{{\rm Dim\,}}

\newtheorem{thm}{Theorem}[section]
\newtheorem{cor}[thm]{Corollary}
\newtheorem{lem}[thm]{Lemma}
\newtheorem{prop}[thm]{Proposition}

\newtheorem {ques/conj}[thm]{Question/Conjecture}

\newtheorem{defn}[thm]{Definition}

\usepackage[numbers]{natbib}

\newcommand{\Irr}{{\rm Irr}}

\newcommand{\cf}{\emph{cf.}~}

\newcommand{\bd}{\boldsymbol}

\begin{document}

\ytableausetup{mathmode, boxframe=normal, boxsize=2.2em}

\renewcommand{\theequation}{\arabic{equation}}
\numberwithin{equation}{section}

\title[Irreducible representations of $\GL_n(\BC)$]{Irreducible representations of $\GL_n(\BC)$ of minimal Gelfand-Kirillov dimension}

\author{Zhanqiang Bai}\author{Yangyang Chen}\author{Dongwen Liu}\author{Binyong Sun}
\address[Bai]{School of Mathematical Sciences, Soochow University, Suzhou 215006,  China}
\email{zqbai@suda.edu.cn}

\address[Chen]{School of Sciences, Jiangnan University, Wuxi, 214122,  China}
\email{chenyy@amss.ac.cn}

\address[Liu]{School of Mathematical Sciences, Zhejiang University, Hangzhou, 310058, China }
\email{maliu@zju.edu.cn}

\address[Sun]{Institute for Advanced Study in Mathematics, Zhejiang University, Hangzhou, 310058,  China}
\email{sunbinyong@zju.edu.cn}


\maketitle

\begin{abstract}
    In this article, by studying the Bernstein degrees and Goldie rank polynomials, we establish a comparison between the irreducible representations of $G=\GL_n(\BC)$ possessing the minimal Gelfand-Kirillov dimension and those induced from finite-dimensional representations of the maximal parabolic subgroup of $G$ of type $(n-1,1)$. We give the transition matrix between the two bases for the corresponding coherent families.
\end{abstract}

\section{Introduction and main results} \label{sec-intro}

Given a real reductive  group $G$, the Langlands classification 
provides a description of $\mathrm{Irr}(G)$, the set of isomorphism classes of irreducible Casselman-Wallach representations of $G$.
 Every representation in $\mathrm{Irr}(G)$  is realized as the unique quotient of a standard representation of $G$, to be called the Langlands quotient. However, the Langlands quotients are somehow mysterious. For example, some fundamental invariants of them such as Gelfand-Kirillov dimensions, Bernstein degrees, and associated cycles, are often quite difficult to calculate. It is thus desirable to have some other realizations of  representations in $\mathrm{Irr}(G)$ that are more concrete than the Langlands quotients. 
In this article, we are concerned with  irreducible representations of $\GL_n(\BC)$ ($n\geq 2$)  that have minimal nonzero Gelfand-Kirillov dimension (which is $2n-2$). In particular, we will show that all such representations occur in the representations that are induced from finite-dimensional representations of the maximal parabolic subgroup of type $(n-1,1)$.


The set $\mathrm{Irr}(\BC^\times)$ is obviously identified with the set of characters of $\BC^\times$, which is further identified with the set 
\[
  \{(a,b)\in \BC\times \BC\mid a-b\in \BZ\} 
\]
so that a pair $(a,b)$ corresponds to the character
\begin{equation}\label{chiab}
 \begin{array}{rcl}
 \chi_{a,b}: \BC^\times &\rightarrow& \BC^\times, \\ z&\mapsto &z^a \bar z^b=(z\bar z)^a \cdot\bar z^{b-a}\qquad (\bar z\textrm{ denotes the complex conjugation of $z$}).
 \end{array}
\end{equation}

From now on we suppose that $G=\GL_n(\BC)$ ($n\in \BN:=\{0,1,2,\cdots\}$). A Langlands parameter for $G$ is a multiset of cardinality $n$ whose elements all belong to 
$\mathrm{Irr}(\BC^\times)$. 
Write $\CP(G)$ for the set of Langlands parameters for $G$. Then the Langlands correspondence, which in the case of complex groups is due to Zhelobenko (\cite[Chapter 7]{Z74}), gives a bijection 
\[
  \CP(G)\rightarrow \Irr(G), \qquad \gamma\mapsto \overline X(\gamma).
\]

For $a,b\in \BC$, we write
\[
  a\succ b\quad \textrm{if } a-b\in \BN^*:=\{1,2,3, \cdots\}.
\]
Suppose that $\gamma$ is a finite multiset whose elements all belong to $\mathrm{Irr}(\BC^\times)$. We say that $\gamma$ is totally  ordered if 
it is multiplicity free and has the form
\[
  \gamma=\{(a_1, b_1), (a_2, b_2), \cdots, (a_k, b_k)\}\qquad (k\geq 0)
\]
such that $a_1\succ a_2\succ\cdots \succ a_k$ and $b_1\succ b_2\succ \cdots \succ b_k$. We say that $\gamma$ is integral if for all $(a_1,b_1),(a_2,b_2)$ in $\gamma$, $a_1-a_2\in \BZ$ (or equivalently, $b_1-b_2\in \BZ$). In particular, a totally ordered multiset $\gamma$ is integral. 

In the rest of the article we assume that $n\geq 2$. Let $\gamma\in \CP(G)$. Note that if $\overline X(\gamma)$ is finite dimensional, then its Gelfand-Kirillov dimension is $0$, and otherwise it has Gelfand-Kirillov dimension $\geq 2n-2$. We will prove the following result in Section \ref{sec2}.

\begin{prop}\label{thmin}
(a) The representation $\overline{X}(\gamma)$ is finite dimensional if and only if $\gamma$ is totally ordered. 

(b) The representation  $\overline{X}(\gamma)$ has Gelfand-Kirillov dimension $2n-2$ if and only if $\gamma$ is not totally ordered and has a submultiset of cardinality $n-1$ that is  totally ordered. 
\end{prop}

By Proposition \ref{thmin}, the irreducible representation  $\overline{X}(\gamma)$ has Gelfand-Kirillov dimension $\leq 2n-2$ if and only if $\gamma$ has a submultiset of cardinality $n-1$ that is  totally ordered. Assume this is the case, and write  \[
\gamma=\gamma'\cup\{(a_0, b_0)\},
\] 
where   $\gamma'=\{(a_1',b_1'),\ldots,(a_{n-1}',b_{n-1}')\}$ is totally ordered ($\gamma$ itself may or may not be totally ordered) with
\[
a'_1\succ a'_2\succ\dots \succ  a'_{n-1}\quad \textrm{ and }\quad  b'_1\succ b'_2\succ\cdots \succ b'_{n-1}.
\]
Put
\begin{equation}\label{ind}
  I_{\gamma}(a_0,b_0):=\mathrm{Ind}_{P_{n-1,1}}^{G} \left(F_{\lambda'}\otimes  \chi_{a_0,b_0}\right)\qquad (\textrm{normalized smooth induction}),
\end{equation}
where $P_{n-1,1}$ is the standard parabolic subgroup of $G$ of type $(n-1,1)$ so that $\GL_{n-1}(\BC)\times \BC^\times$ is a Levi subgroup of it, and $F_{\lambda'}$ is the  irreducible finite-dimensional representation of $\mathrm{GL}_{n-1}(\BC)$ with infinitesimal character 
\[
\lambda':=(a_1', \dots, a_{n-1}', b_1',  \dots, b_{n-1}').
\]
By \cite[Corollary 5.0.10]{B98}, we have the following result. 

\begin{prop} \label{Ind-BD}
The representation $I_{\gamma}(a_0,b_0)$ has Gelfand-Kirillov dimension $2n-2$ and Bernstein degree $\dim F_{\lambda'}$.
\end{prop}

The following result is a special case of \eqref{irred} in the proof of Proposition \ref{thmin} (b) given in Section \ref{sec2}.

\begin{thm}\label{nonint}
Suppose that $\gamma$ is not integral. Then 
\[
\overline X(\gamma) \cong I_{\gamma}(a_0,b_0).
\]
\end{thm}



From now on, we further assume that $\gamma$ is integral. Write
\[
\Lambda:=\Lambda_1^n\times \Lambda_1^n\subset \BC^n\times \BC^n,\quad\textrm{where  $\Lambda_1:=a_0+\BZ=b_0+\BZ$.}
\]
 Write $\mathcal K(G)$ for the Grothendieck group (with $\BC$-coefficients) of the category of Casselman-Wallach representations of $G$. For every Casselman-Wallach representation $V$ of $G$, write $[V]\in \mathcal K(G)$ for the Grothendieck group element represented by $V$. Write $\mathcal R(G)$ for the subspace of $\mathcal K(G)$ spanned by all elements of the form $[F]$, where $F$ is an irreducible finite-dimensional representation of $G$ that is integral in the sense that it extends to a holomorphic representation of $G\times G$ with respect to the complexification map 
\begin{equation}\label{coml}
G\rightarrow G\times G, \ g\mapsto (g,\bar g).
\end{equation}
Here $\bar g$ denotes the entry-wise complex conjugation of $g$.
By using the tensor products of representations, $\mathcal R(G)$ is naturally a commutative $\BC$-algebra, and $\mathcal K(G)$ is naturally a $\mathcal R(G)$-module. 

A $\mathcal K(G)$-valued coherent family on $\Lambda$ is a map
\[
 \Psi: \Lambda\rightarrow \mathcal K(G) 
\]
such that
\begin{itemize}
    \item for all $\lambda\in \Lambda$, $\Psi(\lambda)$ is a linear combination of irreducible representations of $G$ with infinitesimal character $\lambda$;
    \item for every irreducible finite-dimensional representation  $F$ of $G$ that is integral, we have
\[
  [F]\cdot \Psi(\lambda)=\sum_\mu \Psi(\lambda+\mu)\qquad\textrm{for all $\lambda\in \Lambda$},
\]
where $\mu$ in the summation runs over all weights of $F$, counted with multiplicities.
\end{itemize}

Write $\mathrm{Coh}_\Lambda(\mathcal K(G))$ for the space of all $\mathcal K(G)$-valued coherent families on $\Lambda$. Recall that an element 
$(a_1, a_2, \dots, a_n, b_1,b_2, \dots, b_n)\in \Lambda$ is said to be regular (for $G$) if  $a_1$, $a_2$, $\dots$, $a_n$ are pairwise distinct, and $b_1$, $b_2$, $\dots$, $b_n$ are also pairwise distinct. It is said to be dominant (for $G$)  if $a_i-a_j, b_i-b_j \notin -\BN^*$ for all $1\leq i<j\leq n$.
Let $\mathfrak S_n$ denote the permutation group of the set $\{1,2, \dots, n\}$. 
For every $\lambda:=(a_1, a_2, \dots, a_n, b_1,b_2, \dots, b_n)\in \Lambda$ and every $w\in \mathfrak S_n$, write
\begin{equation} \label{gammaw}
  \gamma_{\lambda, w}:=\{(a_1,b_{w^{-1}(1)}),(a_2,b_{w^{-1}(2)}),\dots, (a_n,b_{w^{-1}(n)})\},
\end{equation}
which is a Langlands parameter for $G$.
By \cite[Chapter 7]{V81}, we have the following result.


\begin{prop}  \label{base}
For every $w\in \mathfrak S_n$, there is a  unique element $\overline \Psi_w\in \mathrm{Coh}_\Lambda(\mathcal K(G))$ such that for all regular dominant element $\lambda\in \Lambda$, $\overline \Psi_w(\lambda)=[\overline X(\gamma_{\lambda,w})]$, where $\gamma_{\lambda,w}$ is defined as in \eqref{gammaw}.
 Moreover, $\{\overline \Psi_w\}_{w\in \mathfrak S_n}$ is a basis of  $\mathrm{Coh}_\Lambda(\mathcal K(G))$. 
\end{prop}

For $i, j\in \{1, 2, \ldots, n\}$, define the following cycle in $\mathfrak S_n$
\begin{equation}\label{cyc}
w_{i,j} := \begin{cases} (i (i-1)\cdots j), & \textrm{if }i>j,   \\  (i(i+1)\cdots j), & \textrm{if }i<j, \\
1, & \textrm{if }i=j.
\end{cases}
\end{equation}
Note that $w_{i,i+1}= w_{i+1, i}$ for all $1\leq i<n$, and  $w_{i,j}^{-1}=w_{j,i}$ for all $1\leq i,j\leq n$. For every  $w\in  \mathfrak S_n$, let $\overline \Psi_w\in \mathrm{Coh}_\Lambda(\mathcal K(G))$ be as in Proposition \ref{base}. 
We have the following consequence of Proposition \ref{thmin}.
 
 \begin{cor}\label{cohm}
Let $w\in  \frak S_n$ and let $\lambda\in \Lambda$ be a regular dominant element. Then the irreducible representation    $\overline{\Psi}_w(\lambda)$ has Gelfand-Kirillov dimension $2n-2$  if and only if $w=w_{i,j}$ for some distinct $i, j\in \{1, 2, \ldots, n\}$.
 \end{cor}

 For $ i, j\in\{2,\ldots, n\}$, put
\begin{equation} \label{v_ij}
v_{ i, j }:= \begin{cases}  w_{i-1, j}, & \textrm{if }i \leq j, \\
w_{i, j-1}, & \textrm{if } i>j.
\end{cases}
\end{equation}
It is easy to see that
\[
\{v_{i,j} \mid  i, j \in \{2, \ldots,  n\}\} = \{w_{i,j} \mid  i, j \in \{1,2,\ldots,n \}, i\neq j\}.
\]

Let ${\rm Coh}_\Lambda^{\min}(\CK(G))$ be the subspace of ${\rm Coh}_\Lambda(\CK(G))$ consisting of the coherent families $\Psi$ such that for all $\lambda\in \Lambda$, $\Psi(\lambda)$ is a linear combination of irreducible representations of Gelfand-Kirillov dimension $\leq 2n-2$. By Corollary \ref{cohm}, \cite[Corollary 7.3.23]{V81} and the Jantzen-Zuckerman translation principle (\cf \cite[Lemma 7.2.15]{V81}),  the space ${\rm Coh}_\Lambda^{\min}(\CK(G))$ has dimension $(n-1)^2+1$, and 
\begin{equation} \label{basis1}
\{\overline{\Psi}_{v_{i,j}} \mid i, j\in \{2,\ldots, n\}\} \cup \{\overline{\Psi}_1\}
\end{equation}
form a basis 
of it.

 Recall that by \cite[Corollary 7.3.23]{V81}, the coherent family $\overline \Psi_w$ ($w\in \frak S_n$) is uniquely determined  by its value at a  dominant element of $\Lambda$, provided that this value is nonzero (hence irreducible).
 We will prove the following result in Section \ref{sec5}.

\begin{thm} \label{thm:sing}
Let $i, j$ be distinct elements in $\{1, 2,\ldots, n\}$. Define $i_0, j_0\in \{1, 2,\ldots, n-1\}$ by
\[
(i_0, j_0) := \begin{cases} (i, j-1), & \textrm{if }i<j, \\
(i-1, j), & \textrm{if }i>j.
\end{cases}
\]
Assume that $\frac{n}{2}\in \Lambda_1$. Put
\[
\lambda^0_{i,j} := \left(\frac{n}{2}-1, \frac{n}{2}-2, \ldots, \frac{n}{2}-i_0, \frac{n}{2}-i_0, \ldots, 1-\frac{n}{2}, \frac{n}{2}-1, \ldots, \frac{n}{2}-j_0, \frac{n}{2}-j_0, \ldots, 1-\frac{n}{2}\right)\in \Lambda
\]
and
\[
\gamma^0_{i,j} := \left\{ \left(\frac{n}{2}-1, \frac{n}{2}-1\right),\left(\frac{n}{2}-2, \frac{n}{2}-2\right)\ldots, \left(1-\frac{n}{2}, 1-\frac{n}{2}\right), \left(\frac{n}{2}-i_0,  \frac{n}{2}-j_0\right)\right\}\in \CP(G).
\]
Then
\[
\overline \Psi_{w_{i,j}}(\lambda^0_{i,j}) =
\left[\overline X(\gamma^0_{i,j})\right]= \left[\Ind^G_{P_{n-1,1}}\left(1\otimes \chi_{\frac{n}{2}-i_0, \frac{n}{2}-j_0}\right)\right],
\]
where $1$ denotes the trivial representation of $\GL_{n-1}(\BC)$.
\end{thm}

Theorem \ref{thm:sing} specifies the value of the coherent family $\overline\Psi_{w_{i,j}}$ at the dominant element $\lambda^0_{i,j}\in \Lambda$, hence determines the whole family $\overline \Psi_{w_{i,j}}$, as mentioned above. 


Now we will introduce another basis of the space ${\rm Coh}_\Lambda^{\min}(\CK(G))$. 
Write $L:=\mathrm{GL}_{n-1}(\BC)\times \BC^\times$, to be viewed as a Levi factor of the standard parabolic subgroup $P_{n-1,1}$. Similarly, we have the Grothendieck group $\mathcal K(L)$ and the space  $\mathrm{Coh}_\Lambda(\mathcal K(L))$  of $\mathcal K(L)$-valued coherent families. The normalized parabolic induction induces a linear map
\[
 \mathrm{Ind}: \mathcal K(L)\rightarrow \mathcal K(G),
\]
which further induces a linear map
\begin{equation}\label{Ind}
 \mathrm{Ind}: \mathrm{Coh}_\Lambda(\mathcal K(L))\rightarrow \mathrm{Coh}_\Lambda(\mathcal K(G)).
\end{equation}
Let $\Theta\in \mathrm{Coh}_\Lambda(\CK(L))$ be the unique coherent family such that if $\lambda\in \Lambda$ is regular dominant (for $L$), then $\Theta(\lambda)$ is the class of the irreducible finite-dimensional representation of $L$ with infinitesimal character $\lambda$. 

For $i,j\in\{1,2, \ldots, n\}$ and $\lambda=(a_1, \ldots, a_n, b_1, \ldots, b_n)\in \Lambda$, write
\[
\lambda_{i,j} := (a_1, a_2, \dots, a_{i-1}, a_{i+1},\dots, a_n, a_i, b_1,b_2, \dots, b_{j-1}, b_{j+1},\dots,b_n, b_j).
\]
Define a map $\Psi_{i,j}: \Lambda \to \CK(G)$ by
\begin{equation} \label{Psi-def}
\Psi_{i,j}(\lambda) := \Ind(\Theta)(\lambda_{i,j})= \Ind \left(\Theta(\lambda_{i,j})\right),
\end{equation}
where $\Ind(\Theta)\in\mathrm{Coh}_\Lambda(\CK(G))$ is given by \eqref{Ind}.
It is easy to see that $\Psi_{i,j}\in \mathrm{Coh}_\Lambda(\CK(G))$, and for regular dominant $\lambda\in \Lambda$ we have that 
\[
\Psi_{i,j}(\lambda)=\left[I_{\gamma_{\lambda, w_{i,j}}}(a_i, b_j)\right]=\left[\mathrm{Ind}_{P_{n-1,1}}^{G} \left(F_{i,j}\otimes  \chi_{a_i,b_j}\right)\right],
\]
where  $\gamma_{\lambda, w_{i,j}}$ is defined as in  \eqref{gammaw}, and $F_{i,j}$ is the irreducible finite-dimensional representation of $\mathrm{GL}_{n-1}(\BC)$ with infinitesimal character 
$
(a_1, a_2, \dots, a_{i-1}, a_{i+1},\dots, a_n, b_1,b_2, \dots, b_{j-1}, b_{j+1},\dots,b_n).
$


For $i\in\{1,\ldots, n\}$, put
\[
[ i  ] := \{i,i+1\}\cap\{2,\ldots, n\} = \begin{cases}
\{2\}, & \textrm{if }i=1, \\
\{i, i+1\}, & \textrm{if }2\leq i<n, \\
\{n\}, & \textrm{if }i=n. 
\end{cases}
\]


 Now we can formulate the main result of this article, which gives another basis of the space ${\rm Coh}_\Lambda^{\min}(\CK(G))$ that is more concrete than the basis \eqref{basis1}.

\begin{thm} \label{thm:coh}
(a) For $i,j\in\{1,\ldots, n\}$, the equality
\[
  \Psi_{i,j}= \sum_{(k,l)\in  [i] \times [j]} \overline \Psi_{v_{k,l}} + \delta_{i,j}\cdot \overline \Psi_1
\]
holds in $\mathrm{Coh}_\Lambda(\mathcal K(G))$,  where $\delta_{i,j}$ is the Kronecker symbol.

(b) For $i, j\in \{2, \ldots, n\}$, 
\[
\overline{\Psi}_{v_{i,j}} = \sum_{i\leq k\leq n, \, j\leq l \leq n}(-1)^{k- i + l - j}\Psi_{k, l}+(-1)^{i+j-1}(n-\max\{i,j\}+1)\overline{\Psi}_1.
\]

(c) For $i, j\in \{1,\ldots, n\}$,
\begin{align*}
\Psi_{1,j} & = \sum_{2\leq k\leq n}(-1)^k \Psi_{k,j}  + (-1)^{j-1} \overline{\Psi}_1, \\
\Psi_{i,1} & = \sum_{2\leq l\leq n}(-1)^l \Psi_{i,l}  + (-1)^{i-1} \overline{\Psi}_1.
\end{align*}

(d) The coherent families $\Psi_{i,j}$, $i, j\in \{2,\ldots, n\}$ and $\overline{\Psi}_1$ form a basis of ${\rm Coh}_\Lambda^{\min}(\CK(G))$.
\end{thm}


By using  Proposition \ref{thmin}, Theorem  \ref{nonint}, 
Theorem \ref{thm:coh} (b), \cite[Corollary 7.3.23]{V81} and the Jantzen-Zuckerman translation principle, we have the following corollary.

\begin{cor}
Every irreducible Casselman-Wallach representation of 
$G$ with minimal nonzero Gelfand-Kirillov dimension occurs in the composition series of some representations that are  induced   from  finite-dimensional representations of $P_{n-1,1}$ (via normalized smooth induction). 
\end{cor}


In Proposition \ref{cor:d}, we will determine the Bernstein degrees of the representations $\overline X(\gamma_{\lambda,w_{i,j}})$ for regular dominant $\lambda\in \Lambda$ and distinct $i,j\in \{1,2,\cdots, n\}$. 
This combined with the parabolic Kazhdan-Lusztig theory, enables us to prove Theorem \ref{thm:coh}.

In general, the transition matrices between irreducible representations and standard representations in terms of Kazhdan-Lusztig polynomials are very complicated (see \cite{V81, V82}). For irreducible representations with minimal Gelfand-Kirillov dimension and the parabolically induced representations considered in this article, the transition matrices are much simpler.
It is clear that our method also works for highest weight modules and Verma modules of Lie algebras, and we refer the readers to \cite[Chapter 8]{H08} for their Kazhdan-Lusztig theory.

We will prove Theorem \ref{thm:coh} through Sections \ref{sec3}--\ref{sec5} in this article. Our method is to study the Bernstein degrees and Goldie rank polynomials, which seems to be new, and our results  partially extend the previous studies of degenerate principal series representations  which use  different tools (\cf \cite{HL99, A12, G17, X18}).

\section{Proof of Proposition \ref{thmin}} \label{sec2}
This section is devoted to the proof of Proposition \ref{thmin}, which will not be used in  other parts of this article.

Part (a) is well-known from the Langlands classification. Since $\overline{X}(\gamma)$ is the irreducible subquotient of the principal series representation containing the minimal $K$-type ($K:=\mathrm U(n)$, ``minimal" in the sense of Vogan, \cf \cite{V79b}), the ``if" part of (b) is implied by (a).  Let us prove the ``only if" part of (b). 

Write
\[
\gamma = \gamma_1\cup\cdots\cup \gamma_k
\]
as a disjoint union of multisets such that each $\gamma_i$ is integral but $\gamma_i \cup \gamma_j$ is not integral for any $i\neq j$. Let $n_i := |\gamma_i |$ (the cardinality) and assume that $n_1\geq\cdots \geq n_k$. We first prove that if $\Dim \overline{X}(\gamma)=2n-2$, then $n_1\geq n-1$. Here and after, Dim stands for the Gelfand-Kirillov dimension. 

Let $P :=P_{n_1,\ldots, n_k}$ be the standard parabolic subgroup of $G$ of type $(n_1,\ldots, n_k)$
with Levi subgroup $L := \GL_{n_1}(\BC)\times\cdots\times \GL_{n_k}(\BC)$. Write 
\[
\overline{X}_L(\gamma):=\overline{X}(\gamma_1)\widehat\otimes\cdots\widehat\otimes \overline{X}(\gamma_k)
\quad (\textrm{the completed projective tensor product}),
\]
which is an irreducible representation of $L$.
It follows from the Kazhdan-Lusztig-Vogan duality (\cite{V82}) that
\begin{equation} \label{irred}
\overline{X}(\gamma) \cong \Ind^G_{P} \left( \overline{X}_L(\gamma) \right). 
\end{equation}

Let us sketch a proof of this fact. We first reduce \eqref{irred} to the case that $\overline{X}(\gamma)$ 
has regular infinitesimal character. For each $1\leq i\leq k$, there exists a unique pair consisting of a permutation  $w_i\in \FS_{n_i}$ and a dominant element $\lambda_i=(\bd{a}_i, \bd{b}_i) \in \BC^{n_i} \times \BC^{n_i}$ such that
\begin{itemize}
    \item $\gamma_i$ is of the form $\gamma_{\lambda_i, w_i}$ in the sense of \eqref{gammaw};
    \item  $w_i$ is  the minimal length element of the double coset $\FS_{n_i, \bd{a}_i} w_i \FS_{n_i, \bd{b}_i} $, where $\FS_{n_i, \bd{a}_i}$ and $\FS_{n_i, \bd{b}_i}$ are the stabilizers of $\bd{a}_i$ and $\bd{b}_i$ in $\FS_{n_i}$ respectively.
\end{itemize}
Write 
\[
w:=(w_1,w_2, \dots,w_k)\in \FS_{n_1}\times \FS_{n_2}\times \dots\times \FS_{n_k}\subset \FS_n
\]
and 
\[
 \lambda:=(\bd{a}_1, \bd{a}_2, \dots, \bd{a}_k, \bd{b}_1, \bd{b}_2, \dots, \bd{b}_k)\in \BC^n\times \BC^n. 
\]
 Put 
\[
\Lambda^\prime:= \lambda+\BZ^n\times \BZ^n \subset \mathbb{C}^n\times \mathbb{C}^n=(\mathbb{C}^{n_1}\times \mathbb{C}^{n_1})\times\dots \times (\mathbb{C}^{n_k}\times \mathbb{C}^{n_k}).
\]
As before, we have a space ${\rm Coh}_{\Lambda^\prime}(\CK(L))$ of coherent families on $\Lambda^\prime$, and a linear map 
\[
\Ind: {\rm Coh}_{\Lambda^\prime}(\CK(L))\rightarrow {\rm Coh}_{\Lambda^\prime}(\CK(G))
\]
induced by the parabolic induction. 

Let $\overline{\Psi}_{L, w}\in {\rm Coh}_{\Lambda^\prime}(\CK(L))$ be the coherent family such that 
\[
\overline{\Psi}_{L, w}(\lambda') = [ 
 \overline{X}(\gamma_{\lambda_1',w_1})\widehat\otimes \cdots\widehat\otimes \overline{X}(\gamma_{\lambda_k',w_k})]
\]
for every   element $\lambda'=(\lambda'_1, \dots, \lambda'_k)\in \Lambda^\prime$ with each $\lambda'_i\in \mathbb{C}^{n_i}\times \mathbb{C}^{n_i}$ regular and dominant, where $\gamma_{\lambda'_i, w_i}$ is defined as in \eqref{gammaw}. By \cite[Theorem 6.18]{SV80} (\cf \cite[(8.7.3)]{V81}) and the description of $w_i$, $1\leq i\leq k$, we have that 
\[
\overline{\Psi}_{L, w} (\lambda)  = [\overline{X}_L(\gamma)]\quad \textrm{and} \quad \overline{\Psi}_w(\lambda) = [\overline{X}(\gamma)].
\]
Thus to prove \eqref{irred}, we only need to show that $\Ind(\overline{\Psi}_{L, w}) = \overline{\Psi}_w$, for which it suffices to prove \eqref{irred} when $\lambda$ is regular. 

Assume now that $\lambda$ is regular. By \cite[Theorem 10.1]{V82}, the blocks $B$ (\cite[Definition 1.14]{V82}) and  $B_L$ of $\overline{X}(\gamma)$ and $\overline{X}_L(\gamma)$ respectively, are in duality with the same block $\check B$ of another reductive group $\check G$.  Moreover, we identify $B$ and $B_L$ with a set of regular characters $\phi$ of the diagonal torus (see \cite[Definition 6.6.1]{V81} and \cite[(2.17)]{V82}), and write $X(\phi)$ and $X_L(\phi)$ for the standard representations of $G$ and $L$ respectively. By \cite[Theorem 13.13]{V82}, we have the equality of multiplicities (see \cite[(1.1) and (1.2)]{V82} for the notation)
\begin{equation}\label{eq:mult}
M(X(\phi), \overline{X}(\gamma) ) = \epsilon_{\gamma,\phi}\cdot m( \overline{X}_{\check G}(\check\gamma), X_{\check G}(\check\phi) ) = M( X_L(\phi), \overline{X}_L(\gamma))
\end{equation}
for some $\epsilon_{\gamma,\phi}=\pm1$.
Working within the Grothendieck group, it follows that
\begin{align*}
\left[\Ind^G_{P} \left(\overline{X}_L(\gamma)\right)\right] & = \Ind^G_P \left(\sum_{\phi\in B_L}
M(X_L(\phi), \overline{X}_L(\gamma))[ X_L(\phi)]\right) \\
& = \sum_{\phi\in B_L}M(X_L(\phi), \overline{X}_L(\gamma)) [\Ind^G_P(X_L(\phi))] \\
& = \sum_{\phi\in B}M(X(\phi), \overline{X}(\gamma) ) [X(\phi)] \\
& = [\overline{X}(\gamma)],
\end{align*}
where the penultimate equality follows from \eqref{eq:mult} and induction in stages. 
This proves \eqref{irred}.

From \eqref{irred} and \cite[Corollary 5.0.10]{B98}, it follows that 
\[
\Dim \overline{X}(\gamma) = \sum^k_{i=1} \Dim \overline{X}(\gamma_i) + \dim_\mathbb{R} (\frak{g}/\frak{p}), 
\]
where $\frak{g}$ and $\frak{p}$ are the Lie algebras of $G$ and $P$ respectively. If $\Dim \overline{X}(\gamma)=2n-2$, then $\dim_\BR(\frak{g}/\frak{p})\leq 2n-2$, which forces that $n_1\geq n-1$. 

If $n_1=n-1$ so that $\gamma$ is not integral, then $\Dim \overline{X}(\gamma) = 2n-2 = \dim_\BR(\frak{g}/\frak{p})$ if and only if $\Dim \overline{X}(\gamma_1)=0$, which is equivalent to that $\gamma_1$ is totally ordered by part (a).

If $n_1=n$, then $\gamma$ is integral, in which case the proposition follows easily from \cite[Theorem 3.20]{BV85}.

\section{Bernstein degree} \label{sec3}

\subsection{Primitive ideals} We recall some generalities from \cite{BV85}, applied to the group
 $G=\GL_n(\BC)$ ($n\geq 2$). Recall that $\Fg=\Lie(G)$, whose complexification is \begin{equation}\label{coml5}
 \Fg_\BC=\Fg\times \Fg
 \qquad \textrm{as in \eqref{coml}}, 
  \end{equation}
 with  universal enveloping algebra
\[
\CU(\Fg_\BC) = \CU(\Fg)\otimes_\BC \CU(\Fg).
\]
Assume that
\[
\lambda= (\bd{a},\bd{b}):= (a_1,\ldots, a_n, b_1,\ldots, b_n)\in \Lambda
\]
is regular dominant, where $\Lambda$ is defined as in the Introduction.  Fix the standard upper triangular Borel subgroup $B$, and the maximal compact subgroup $K=\RU(n)$ of $G$. For each $w\in \frak S_n$,  let 
\[
X(\gamma_{\lambda,w}):=\Ind^G_B(\chi_{a_1, b_{w^{-1}(1)}}\otimes \cdots\otimes \chi_{a_n,  b_{w^{-1}(n)}})
\]
be the principal series representation of $G$, where $\gamma_{\lambda,w}$ is defined in \eqref{gammaw}. Then  $\overline X(\gamma_{\lambda,w})$ is the Langlands subquotient of $X(\gamma_{\lambda,w})$, that is, the unique irreducible subquotient containing the $K$-type of extremal weight
\begin{equation} \label{ex-wt}
\bd{a} - w \bd{b}:=(a_1 -b_{w^{-1}(1)},\ldots, a_n - b_{w^{-1}(n)}).
\end{equation}

For $w\in \FS_n$, let $V^L(w)$ be the  left cell representation of $\FS_n$ defined by  \cite[(3.17)]{BV85}. Then $V^L(w)$ is irreducible in this case, and the Springer correspondence (\cf \cite[Theorem 2.18]{BV85}) associates to it a nilpotent orbit $\CO(w)$ in $\frak{g}^* := \Hom_\BC(\Fg,\BC)$. For a general nilpotent orbit $\CO$ in $\Fg^*$, write $\CO_{\BR}:=(\CO\times \CO)\cap \sqrt{-1}\Fg_\BR^*$, which is a   nilpotent $G$-orbit in $\sqrt{-1}\Fg_\BR^*$.  Here $\Fg_\BR^*:=\Hom_\BR(\Fg,\BR)$, which is identified with a real form of $\Fg^*\times \Fg^*$ via the identification \eqref{coml5}.
Let $w_0$ be the longest element of $\frak S_n$. 
By \cite[Theorem 3.20]{BV85} (see also 
\cite{BV82}), the wavefront set of  the representation $\overline{X}(\gamma_{\lambda,w})$ is the topological closure $\overline{\CO(ww_0)_{\BR}}$ of $\CO(ww_0)_{\BR}$ in $\sqrt{-1}\Fg^*_\BR$.

Let $\frak{h}$ be the standard Cartan subalgebra of $\Fg$, which is canonically identified with $\BC^n$. Let $\Phi(\Fg, \Fh)$ be the corresponding root system, in which a positive subsystem is determined by $\Fb :=\Lie(B)$. For any $\nu\in \Fh^*$, which is identified with $\Fh$ under the standard bilinear form $\langle \ ,\rangle$ on $\Fh$, 
let $L(\nu)$ be the unique irreducible quotient of the Verma module 
$
M(\nu):= \CU(\Fg)\otimes_{\CU(\Fb)}\BC_{\nu-\rho},
$
where $\rho$ is the half sum of the positive roots in $\Phi(\Fg, \Fh)$,
and denote by $I(\nu)$ the primitive ideal
(namely,  annihilator ideals of  irreducible $\CU(\Fg)$-modules)
\[
I(\nu):={\rm Ann}_{\CU(\Fg)} L(\nu). 
\]
By \cite[Theorem 5.2]{J77}, the annihilator of $\overline{X}(\gamma_{\lambda,w})$ in $\CU(\Fg_\BC)$ is 
\[
\check I (-  w_0w^{-1} \bd{a}) \otimes \CU(\Fg) + \CU(\Fg)\otimes \check I(- w_0 w \bd{b}),
\]
where  $\check{\,}$ is the anti-involution on $\CU(\Fg)$ induced by $-1$ on $\Fg$. Then for $x,  y\in \frak S_n$, 
\[
\check I (- w_0x^{-1} \bd{a}) =\check  I (-  w_0y^{-1} \bd{a}) \quad \Longleftrightarrow \quad  xw_0  \sim_R  yw_0 
\]
and likewise, 
\[
\check I (- w_0 x \bd{b} ) =\check  I (-   w_0 y \bd{b} ) \quad \Longleftrightarrow \quad w_0 x \sim_L w_0 y \quad \Longleftrightarrow \quad xw_0 \sim_L yw_0,
\]
where $\sim_L$ and $\sim_R$ are the equivalence relations in \cite{KL79} that define the  left cells and  right cells respectively. 
For the last equivalence, we note that 
$ w\mapsto w_0w w_0$ induces an involution on the left cells (\cf \cite[Corollary 6.2.10]{BB05}).

For a Harish-Chandra module $X$ of $G$, denote by $c(X)$ the Bernstein degree of $X$ (see \cite[Theorem 1.1]{V78}). By \cite[6.1]{J80b}, up to a nonzero rational number,  $c(\overline{X}(\gamma_{\lambda,w}))$ is given by
\begin{equation} \label{bd}
c_w(\lambda):= {\rm rk} \left( \CU(\Fg)/I( -  w_0w^{-1} \bd{a}) \right)\cdot {\rm rk}\left( \CU(\Fg)/I(-w_0 w \bd{b})\right),
\end{equation}
where rk stands for the Goldie rank. That is,
\begin{equation} \label{const}
c(\overline{X}(\gamma_{\lambda,w})) = d_w\cdot c_w(\lambda)
\end{equation}
for some constant $d_w\in \mathbb{Q}^\times$, which is independent of the regular dominant element $\lambda = (\bd{a}, \bd{b})\in \Lambda$.
Here by abuse of notation, $$c(\overline{X}(\gamma_{\lambda,w})):=c(\overline{X}(\gamma_{\lambda,w})^{K-fin})$$ where $\overline{X}(\gamma_{\lambda,w})^{K-fin}$ denotes the space of $K$-finite vectors of $\overline{X}(\gamma_{\lambda,w})$, which is naturally a Harish-Chandra module.

\subsection{Induced representations}  \label{ssec:IR}
Recall that the maximal parabolic subgroup $P_{n-1,1}$ has Levi subgroup $L = \GL_{n-1}(\BC)\times \BC^\times$. 
Let $\Phi_L^+ = \Phi^+(\Fl, \Fh)$ be the set of positive roots for $\Fl:=\Lie(L)$ with respect to $\Fb\cap\Fl$, and denote by $\rho_L$ its half sum. Let $w_{L}$ be the longest element of $\frak S_{n-1}\subset \frak S_n$, that is,
\begin{equation}\label{wl}
w_{L} = \begin{pmatrix} 1 & 2 & \cdots & n-1 & n \\n-1 &  n-2 & \ldots & 1 & n\end{pmatrix}.
\end{equation}
The set $\frak S^{\underline{L}}$  of   minimal length representatives of the cosets $\frak S_{n-1}\backslash \frak S_n$ is given by 
\[
\begin{aligned}
\frak S^{\underline{L}} &= \{x\in \frak S_n   \mid x^{-1}(\alpha) \in \Phi^+(\Fg, \Fh)\textrm{ for all } \alpha\in \Phi_L^+\}\\
& = \{x_1,\ldots, x_n\},
\end{aligned}
\]
where for $i=1,\ldots, n$,
\begin{equation}\label{mini}
x_i := \begin{pmatrix}  1 & \cdots & i-1 & i & i+1 & \cdots & n \\ 1 & \cdots & i-1 &  n &  i & \cdots & n-1\end{pmatrix}.
\end{equation}
It follows directly that 
\[
x_i^{-1} x_j = w_{i,j},\quad i,j= 1,\ldots, n,
\]
where $w_{i,j}$ was defined in \eqref{cyc}.
The Bruhat order on $\frak S_n$ induces a total order on $\frak S^{\underline{L}}$ so that
\[
w_{L} w_0 = x_1 >\cdots > x_n =1.
\]
 Define 
 $y_i := w_{L} x_i$, $i=1,\ldots, n$,
 so that $\frak S^{\overline L}:=\{y_1,\ldots, y_n\}$ is the set of maximal length representatives of 
 $\frak S_{n-1}\backslash \frak S_n$, and the map 
 \[
 \frak S^{\underline L}\to \frak S^{\overline L},\quad x_i\mapsto y_i
 \]
 is the unique order-preserving bijection.

For $\nu\in\Fh^*$, define
\begin{equation}\label{poly}
h_L(\nu): = \prod_{\alpha\in \Phi_L^+}\frac{\langle\nu, \alpha\rangle}{\langle\rho_L, \alpha\rangle},
\end{equation}
which is naturally a polynomial function on $\Fh^*$. For $i,j\in\{1, 2, \ldots, n\}$ and $\lambda= (\bd{a}, \bd{b})= (a_1,\ldots, a_n, b_1,\ldots, b_n)\in \Lambda$, recall that the induced representation $I_{\gamma_{\lambda,w_{i,j}}}(a_i, b_j)$
was defined as in \eqref{ind} and $F_{i,j}$ denotes the irreducible finite-dimensional representation of $\mathrm{GL}_{n-1}(\BC)$ with infinitesimal character $(a_1, a_2, \dots, a_{i-1}, a_{i+1},\dots, a_n, b_1,b_2, \dots, b_{j-1}, b_{j+1},\dots,b_n)$.

\begin{lem} \label{lem:BD}
With the notations as above, the Bernstein degree of $I_{\gamma_{\lambda,w_{i,j}}}(a_i, b_j)$ is given by 
\[
c(I_{\gamma_{\lambda,w_{i,j}}}(a_i, b_j)) =\dim F_{i,j} = h_L(x_i \bd{a} ) \cdot h_L (x_j \bd{b}) = h_L(- y_i \bd{a}) \cdot h_L(- y_j \bd{b}).
\]
\end{lem}

\begin{proof}
This follows from Proposition \ref{Ind-BD} and the well-known Weyl  dimension formula. The last equality follows from the fact that $w_{L}$ maps $\Phi_L^+$ onto $-\Phi_L^+$. 
\end{proof}

\section{Goldie rank} 

In this section we evaluate $\dim F_{i,j}$ given by Lemma \ref{lem:BD} in terms of  $c_{w_{i,j}}(\lambda)$ as in \eqref{bd}. The notations are as in the Introduction and Section \ref{sec3}. The main result of this section is as follows.

\begin{prop}\label{prop:BD}
For $i, j\in \{1,\ldots, n\}$, 
\[
\dim F_{i, j} = \sum_{(k,l)\in [i]\times [j]}c_{v_{k,l}}(\lambda).
\]
\end{prop}

\subsection{Cells and minimal elements} Firstly we recall some definitions and properties about cells and Young tableaux. The details can be found in \cite{S01} and \cite{B11}. 
 Consider the left cells in $\frak S_n$ whose associated Young diagram via the Robinson-Schensted algorithm has two columns of lengths $n-1$ and $1$ respectively. Denote this Young diagram by $\CD$, whose rows give the partition $[2,1^{n-2}]$ of $n$. The corresponding left cell representation of $\frak S_n$, denoted by
$\sigma$, is the reflection representation of dimension $n-1$. There are $n-1$ such left cells, whose disjoint union forms a double cell $\CC$ and  gives rise to the representation $\sigma  \otimes \sigma$ of $\frak S_n\times\frak S_n$. By direct calculation, 
\[
\CC=\{w_{ij}w_0 \mid  i, j \in \{1, 2,\ldots, n\}, i\neq j\} = \{ w_0w_{ij} \mid  i, j \in \{1, 2,\ldots, n\}, i\neq j\},
\]
where $w_0$ is the longest element of $\mathfrak S_n$ and the cycles $w_{i,j}$ are defined as in \eqref{cyc}.

For $w\in \frak S_n$, let $P(w)$ and $Q(w)$ be the insertion tableau and recording tableau associated to $w$ via the Robinson-Schensted algorithm. The tableaux $P(w)$ and $Q(w)$ are related by $Q(w)=P(w^{-1})$, and for $x, y\in \mathfrak S_n$,
\begin{align*}
& x \sim_L y \quad \Longleftrightarrow \quad Q(x) = Q(y),\\
& x \sim_R y \quad  \Longleftrightarrow \quad P(x)=P(y).
\end{align*}
Here, as mentioned in Section \ref{sec3}, $\sim_L$ and $\sim_R$ are the equivalence relations defined in \cite{KL79}.
Denote by $\CY$ the set of standard Young tableaux (filled with $1, 2, \ldots, n$) with shape $\CD$. It is well-known that
\[
\CC \to \CY \times \CY,\quad w\mapsto (P(w), Q(w))
\]
is a bijection. 

Each left cell in $\CC$ contains a unique  minimal element $w$ \`a la Brundan (\cite{B11}), that is, the two columns of $P(w)$ are labeled by $1,\ldots, n-1$ and $n$ from top to bottom respectively. We denote the set of such minimal elements by $\CC^{\min}$, so that $|\CC^{\min}| = n-1$.
For $w\in \CC$, denote by $w_{\min}\in \CC^{\min}$ the unique minimal element in the left cell containing $w$. 

\begin{lem} \label{lem:min} 
Let $y_1,\ldots, y_n\in \frak S^{\overline L}$ be given as in Section \ref{ssec:IR}. Then the followings hold.
\begin{enumerate}
 \item[(a)] For $1\leq i\neq j\leq n$,
 \[
 (w_0 w_{i,j})_{\min} = \begin{cases} y_j, & \textrm{if }i<j, \\
 y_{j+1}, & \textrm{if }i>j.
 \end{cases}
 \]
\item[(b)]
$\CC^{\min}=\{y_2,\ldots, y_n\}$.
\end{enumerate}
\end{lem}

\begin{proof}
We describe $P(w_0w_{i,j})$ and  $Q(w_0w_{i,j})$ as follows, where  $i, j \in \{1, 2,\ldots, n\}$, $i\neq j$. 

If $i<j$, then $n+1-i\geq 2$, $j\geq 2$, and we have the tableaux
\[
P( w_0 w_{ij} ) = \ytableaushort{1{\scriptstyle n+1-i},2,\vdots,n}
\quad \quad
Q( w_0 w_{ij} ) = \ytableaushort{1 j,2,\vdots,n}.
\]
To be precise,  the first columns in the above two tableaux have $n-1$ boxes, filled with  $\{1, 2, \ldots, n\}\setminus \{n+1-i\}$ and $\{1, 2, \ldots, n\}\setminus \{j\}$ respectively.

Similarly, if $i>j$, then $n+2-i\geq 2$, $j+1\geq 2$, and we have the tableaux
\[
P( w_0 w_{i,j}  ) = \ytableaushort{1{\scriptstyle n+2-i},2,\vdots,n}
\quad \quad
Q( w_0 w_{i,j} ) = \ytableaushort{1 {j+1},2,\vdots,n}.
\]
The lemma follows easily, by noting that $w_0 w_{1,i}=y_i$, $i=1,\ldots, n$.
\end{proof}

It is slightly more convenient to reformulate Lemma \ref{lem:min} (a) in terms of the elements $v_{i,j}$ given by \eqref{v_ij} as follows. 

\begin{cor} \label{cor:min}
For $i, j\in \{2,\ldots, n\}$, it holds that $(w_0 v_{i,j})_{\min} = y_j$.
\end{cor}

\subsection{Goldie rank polynomials}    
For $w\in \frak S_n$, let $\wh{C}_w$ be the set of integral weights lying in the upper closure of the chamber containing $-w\rho$ (see e.g. \cite[Page 1744]{B11} for the precise definition of upper closure). By \cite[Section 5.12]{J80b}, there exists a unique polynomial function  $p_w\in \BC[\Fh^*]$, called the Goldie rank polynomial, such that
\[
{\rm rk}\left(\CU(\Fg)/I(\nu)\right) = p_w(\delta)
\]
for all $\nu \in \widehat{C}_w$, where $\delta$ is the anti-dominant conjugate of $\nu$. Moreover, $p_w$ only depends on the left cell of $w$. It follows that  \eqref{bd} is given by
\begin{equation} \label{bd2}
\begin{aligned}
c_w(\lambda) = p_{(w_0w^{-1})_{\min}}(-\bd{a}) \cdot p_{(w_0w)_{\min}}(-\bd{b}),
\end{aligned}
\end{equation}
where $\lambda=(\bd{a},\bd{b})\in\Lambda$ is regular and dominant.

We recall the formula of $p_w$ (in type A) from \cite{J80a} and \cite[Theorem 1.6]{B11}. 
Assume that $y\in \CC^{\min}$. Then $p_y\in\mathbb C[\Fh^*]$ is given by
\[
p_y(\nu) = \sum_{x\in \frak S^{\overline{L}}}(L_y: M_x) \cdot h_L (x\nu),
\]
where $h_L$ is given as in \eqref{poly},  $L_x:=L(-x\rho)$ denotes the unique irreducible quotient of the Verma module $M_x:=M(-x\rho)$ for each $x\in \frak S_n$, and the multiplicity $(L_y: M_x)\in \mathbb{Z}$ is defined by the equality of virtual representations
\[
[L_y] = \sum_{x\in \frak S_n} (L_y: M_x)\cdot [M_x].
\]
By the Kazhdan-Lusztig Conjecture proven in \cite{BB81, BK81}, 
\[
(L_y: M_x) = (-1)^{\ell(x)+\ell(y)} P_{x, y}(1)\qquad  x, y\in \frak S_n,
\]
where $P_{x, y}(q)\in \BZ[q]$ denotes the Kazhdan-Lusztig polynomial (see \cite{KL79}), and $\ell(x)$ and $\ell(y)$ respectively denote the lengths of $x$ and $y$.  Since $P_{x,y}=0$ if $x\not\leq y$,  we obtain that for $y\in \CC^{\min}$,
\[
p_y(\nu) = \sum_{x\in \frak S^{\overline{L}},\, x\leq y}(-1)^{\ell(x)+\ell(y)}P_{x,y}(1) \cdot h_L(x\nu),
\]
here, as mentioned in Section \ref{sec3}, $\leq$ denotes the Bruhat order on $\mathfrak S_n$.
Recall that $\frak S^{\overline L}=\{y_1>\cdots > y_n\}$ is totally ordered. Thus by Lemma \ref{lem:min} (2), for $i=2,\ldots, n$,
\begin{equation} \label{gr-kl}
p_{y_i}(\nu) = \sum^n_{j=i} (-1)^{\ell(y_j)+\ell(y_i)}P_{y_j,y_i}(1) \cdot h_L(y_j\nu).
\end{equation}

\subsection{Parabolic Kazhdan-Lusztig polynomials}
In view of Lemma \ref{lem:BD}, \eqref{bd2} and \eqref{gr-kl}, we need to evaluate the inverse of the upper triangular matrix 
\[
((-1)^{\ell(y_j)+\ell(y_i)}P_{y_j,y_i}(1))_{i,j=2,\ldots, n},
\]
which is a lower-right submatrix of the upper triangular  matrix 
\begin{equation} \label{mat-C}
((-1)^{\ell(y_j)+\ell(y_i)}P_{y_j,y_i}(1))_{i,j=1,\ldots, n}.
\end{equation}
This can be solved by using the theory of parabolic Kazhdan-Lusztig polynomials introduced in \cite{D87}, the inversion formula in \cite{D90}, and the explicit computation for the maximal parabolic case in
\cite{B02}, which will be recalled below.

Following \cite{D87} and the exposition in \cite{B02}, for each symbol $\star\in\{-1, q\}$, two  families of polynomials (with different notation in \cite{D87})
\[
\{R^{L,\star}_{u,v}(q)\}_{u,v\in \frak S^{\underline{L}}},\quad  \{P^{L,\star}_{u,v}(q)\}_{u,v\in \frak S^{\underline{L}}}\subset \BZ[q]
\]
are introduced, called the  parabolic $R$-polynomials and  parabolic Kazhdan-Lusztig polynomials of $\frak S^{\underline{L}}$ of type $\star$ respectively. We refer to \cite[Section 2]{D87} for the precise definitions. The results we need in this article are the following properties of $P^{L,\star}_{u,v}(q)$:
\begin{itemize}

\item $P^{L, \star}_{u,v}(q)=0$ if $u\not\leq v$;

\item $P^{L,\star}_{u,u}(q)=1$;

\item Relation with the ordinary Kazhdan-Lusztig polynomial:
\[
P^{L,-1}_{u,v}(q) = P_{w_Lu, w_{L}v}(q)\quad(\text{see \eqref{wl} for the definition of $w_L$});
\]

\item Inversion formula (\cite[Theorem 4.6]{D90}):
\[
\delta_{u, v} = \sum_{u\leq x\leq v} (-1)^{\ell(u)+\ell(x)}P_{u, x}^{L,-1}(q) \cdot P^{L,q}_{v^*, x^*}(q),
\]
where $x^* := w_{L}x w_0$ for $x\in \frak S^{\underline L}$, and $\delta_{u,v}$ is the Kronecker symbol.
\end{itemize}
Note that $x\mapsto x^*$ is the unique order-reversing involution of $\frak S^{\underline{L}}$. To be explicit, we have that 
\[
x_i^* = x_{n+1-i},\quad i=1,\ldots, n,
\]
where $x_i$ was defined as in \eqref{mini}.
In view of the above discussions, we conclude that the inverse of the matrix \eqref{mat-C} is given by
\begin{equation} \label{inverse}
 ((-1)^{\ell(y_j)+\ell(y_i)}P_{y_j,y_i}(1))_{i,j=1,\ldots, n}^{-1} = (P^{L, q}_{x_j, x_i}(1))_{i, j=1,\ldots, n}.
\end{equation}

\subsection{Explicit formula} \label{ssec:EF} We will give the explicit formula for $P^{L, q}_{u, v}(q)$, $u, v\in \frak{S}^{\underline{L}}$, following \cite{B02}. 
In order to be self-contained, we recall the general result from {\it loc. cit}. To this end, we need to introduce some terminologies. We identity a partition $[p_1\geq \ldots \geq p_r]$ with its Young diagram 
\[
\{(i, j)\in \BN^*\times \BN^* \mid 1\leq i\leq r, \, 1\leq j\leq p_i\}.
\]
The above elements $(i,j)$ are called the boxes of the Young diagram. 
For a box $(i,j)$, define its  level to be $i+j$. 
In this subsection, we draw the Young diagram of a partition  
rotated counterclockwise by $\frac{3}{4}\pi$ radians with respect to the usual convention.  For example, the following Figure $1$ is the Young diagram of the partition $[3,2]$, where each box is filled by its level:
\begin{figure*}[htb]
{\begin{center}
\rotatebox{145}{\begin{tikzpicture}[scale=\domscale+0.25,baseline=-45pt]
	\hobox{0}{0}{\rotatebox{215}{2}}
	\hobox{1}{0}{\rotatebox{215}{3}}
	\hobox{2}{0}{\rotatebox{215}{4}}
	\hobox{0}{1}{\rotatebox{215}{3}}
    \hobox{1}{1}{\rotatebox{215}{4}}
		\end{tikzpicture}}
	\caption{}
\end{center}}
\end{figure*}

Let $\eta$ be a skew-partition, which means that $\eta$ is a subset of $\BN^*\times \BN^*$ of   the form $\lambda\setminus \mu$ for some Young diagrams  $\mu\subseteq \lambda$. The  outer border  strip $\theta$ of $\eta$ is the subdiagram of $\eta$ such that a box of $\eta$ is in $\theta$ if and only if there is no box of $\eta$ directly above it (for example the outer border  strip of the diagram $[3,2]$ is illustrated in Figure $2$). 
If $\theta$ is connected (by which we mean ``rookwise connected" so that for example the skew-partition $[2,1]\setminus [1]$ is not connected), then it is called a  Dyck cbs (cbs abbreviates ``connected border strip") if no box of $
\theta$ has level strictly less than that of the leftmost or rightmost of its boxes. In particular, the leftmost and the rightmost boxes in a Dyck cbs have the same level. For example, the connected border strip in Figure $2$ is not a Dyck cbs.
\begin{figure*}[htb]
{\begin{center}
 \rotatebox{145}{\begin{tikzpicture}[scale=\domscale+0.25,baseline=-45pt]
	\hobox{1}{0}{\rotatebox{215}{}}
	\hobox{2}{0}{\rotatebox{215}{}}
	\hobox{0}{1}{\rotatebox{215}{}}
    \hobox{1}{1}{\rotatebox{215}{}}
		\end{tikzpicture}}
	\caption{}
\end{center}}
\end{figure*}

Now we define the notion of  Dyck skew-partition, which is crucial to the formula of $P^{L, q}_{u, v}(q)$.

\begin{defn} \label{defn:dyck}
Define a skew-partition $\eta$ to be  Dyck inductively as follows:

\begin{enumerate}
    \item It is Dyck if and only if  all its  connected components are  Dyck;
    \item If $\eta$ is connected, then $\eta$ is Dyck if and only if 
    \begin{enumerate}
        \item its outer border strip $\theta$ is a Dyck cbs, and 
        \item $\eta^{(1)}:=\eta\setminus \theta$ is Dyck;
    \end{enumerate}
    \item The empty partition $\varnothing$ is Dyck.
\end{enumerate}
\end{defn}

As a simple example, the skew-partition $[3,1]\setminus [2]$ is Dyck, but $[3,1]\setminus [1]$ is not Dyck. To see  larger examples, we have that $[4^3, 3]$ is not Dyck, while $[4^4]\setminus[1]$ and $[4^3, 3]\setminus[1]$ are Dyck (\cite{B02}).

Let $\eta$ be a  skew-partition (not necessarily Dyck).  Define its depth ${\rm dp}(\eta)$ inductively by ${\rm dp}(\varnothing)=0$ and 
\[
{\rm dp}(\eta) = c(\theta) + {\rm dp}(\eta^{(1)}),
\]
where $\theta$ is the outer border strip of $\eta$ and $\eta^{(1)}=\eta\setminus \theta$, and $c(\theta)$ denotes the number of connected components of $\theta$. 

Denote by 
$\CP_L(n)$ the set of partitions contained in $[1^{n-1}]$ (including the empty partition). There is a bijection (\cf \cite[Proposition 2.8]{B02}) 
\[
\Psi: \frak S^{\underline{L}}\to \CP_L(n), \quad  x\mapsto \Psi(x):= [1^{i_x}],
\]
where 
\[
i_x:=\textrm{the cardinality of the set } \{r\in \{1,2,\dots, n-1\}\mid x^{-1}(r)>r\}.
\]
Moreover, $u\leq v$ in $\frak S^{\underline{L}}$ if and only if $\Psi(u) \subset \Psi(v)$, in which case 
\begin{equation}  \label{eta}
\eta_{u,v}:=\Psi(v)\setminus \Psi(u)
\end{equation}
is a skew-partition. By \cite[Theorem 5.1]{B02}, it holds that
\[
P^{L, q}_{u,v}(q) = \begin{cases} q^{\frac{1}{2}(|\eta_{u,v}| - {\rm dp}(\eta_{u,v}))}, & \textrm{if }\eta_{u,v}\textrm{ is Dyck,}\\
0, & \textrm{otherwise.}
\end{cases}
\]
In particular, taking $q=1$ gives that
\[
P^{L, q}_{u,v}(1) = \begin{cases} 1, & \textrm{if }\eta_{u,v}\textrm{ is Dyck,}\\
0, & \textrm{otherwise.}
\end{cases}
\]

To be explicit, we have that $\Psi(x_i) = [1^{n-i}]$, $i=1,\ldots, n$. The Dyck condition in this case is very simple:  for $1\leq i\leq j\leq n$,
\[
\eta_{x_j, x_i} = \Psi(x_i)\setminus \Psi(x_j) = [1^{n-i}]\setminus [1^{n-j}],
\]
which is Dyck if and only if $j-i\leq 1$. It follows that 
the matrix  $(P^{L,q}_{x_j, x_i}(1))_{i,j=1,\ldots, n}$ is a Jordan block 
\begin{equation}\label{JD}
   (P^{L,q}_{x_j, x_i}(1))_{i,j=1,\ldots, n} = \begin{pmatrix}
   1 & 1 & 0 & \cdots & 0 & 0  \\
   0 & 1 & 1 & \cdots & 0  & 0 \\
   & & \cdots & \cdots & & \\
   0 & 0 & 0 & \cdots & 1 & 1  \\
   0 & 0 & 0 & \cdots & 0 & 1
   \end{pmatrix}_{n\times n}.
\end{equation}

\subsection{Proof of Proposition \ref{prop:BD}}
By \eqref{gr-kl}, \eqref{inverse} and \eqref{JD}, we find that
for $2\leq i\leq n$, 
\begin{equation} \label{hL1}
h_L( - y_i \bd{a}) =  \begin{cases} p_{y_i}(-\bd{a}) + p_{y_{i+1}}(-\bd{a}), & \textrm{if }2\leq i<n, \\
p_{y_n}(-\bd{a}), & \textrm{if }i=n.
\end{cases}
\end{equation} 
It remains to deal with $h_L(-y_1\bd{a})$.  From the Weyl dimension formula 
\[
h_L(  - y_i\bd{a} ) = \frac{1}{1!2!\cdots (n-2)!} \prod_{1\leq s< t\leq n, \, s, t\neq i} (a_s -a_t),
\]
it is easy to deduce the identity 
\begin{equation} \label{alt}
\sum^n_{i=1} (-1)^i  h_L(-y_i \bd{a})=0.
\end{equation}
This together with \eqref{hL1} gives that
\begin{equation} \label{hL2}
h_L(-y_1 \bd{a}) =  p_{y_2}(-\bd{a}).
\end{equation}
We have similar results for $h_L(-y_j \bd{b})$.

Proposition \ref{prop:BD} is a combination of Lemma \ref{lem:BD}, Corollary \ref{cor:min}, \eqref{bd2}, \eqref{hL1} and \eqref{hL2}.

\section{Proof of the main results} \label{sec5}
This section is devoted to the proofs of Theorems \ref{thm:sing} and \ref{thm:coh}. The notations are as in the Introduction and Section \ref{sec3}.

\subsection{Proof of Theorem \ref{thm:sing}} 
We need to prove that
\[
\overline{\Psi}_{w_{i,j}}(\lambda^0_{i,j}) =  \left[\overline X(\gamma^0_{i,j})\right].
\]
Recall from \cite{X18} the following subsets of $\Irr(\BC^\times)$:
\[
{\rm C}^+(n)=\{\chi_{\frac{n}{2}+r, \frac{n}{2}+s} \mid r, s\in \BN\},\quad 
{\rm C}^-(n)=\{\chi^{-1}  \mid \chi\in {\rm C}^+(n)\}.
\]
Then $\chi_{\frac{n}{2}-i_0, \frac{n}{2}-j_0}\not\in {\rm C}^+(n)\cup {\rm C}^-(n)$. Hence by 
\cite[Theorem 1.1 (a)]{X18},
\[
\Ind^G_{P_{n-1,1}}(1\otimes \chi_{\frac{n}{2}-i_0, \frac{n}{2}-j_0})
\]
is irreducible. Parabolic induction in stages shows that it is a quotient of the principal series representation $X(\gamma^0_{i,j})$. By considering its $K$-types, we see that 
\[
\Ind^G_{P_{n-1,1}}(1\otimes \chi_{\frac{n}{2}-i_0, \frac{n}{2}-j_0})\cong \overline X(\gamma^0_{i,j}).
\]

We can find a regular dominant element $\lambda\in \Lambda$ such that $\overline X(\gamma^0_{i,j})$ is a translation of $\overline{X}(\gamma_{\lambda,w_{i,j}})$, where $\gamma_{\lambda,w_{i,j}}$ is associated with $\lambda$ via \eqref{gammaw}. See \cite[Section 3]{V79} for details. By \cite[Theorem 6.18]{SV80}, we have that
$\overline \Psi_{w_{i,j}}(\lambda^0_{i,j})\neq 0$, in which case it equals  $\left[\overline X(\gamma^0_{i,j})\right]$. This can be made more explicit as follows. Denote by $e_1,\ldots, e_n$ the standard basis of $\frak h^*$, and $\alpha_k := e_k-e_{k+1}$, $k=1,\ldots, n-1$, which consists of the simple roots of $\Phi^+(\frak g, \frak h)$. By the definition of 
$(i_0, j_0)$ we can check that
\[
w_0 w_{i,j}^{-1}(\alpha_{i_0})>0,\quad   w_{i,j}w_0(\alpha_{j_0})>0.
\]
By \cite[(3.9)]{BV85}, $\alpha_{i_0}$ (resp. $\alpha_{j_0}$) does not belong to the left (resp. right) Borho-Jantzen-Duflo $\tau$-invariant of $\overline X(\gamma_{\lambda,w_{i,j}})$. Thus 
$\overline \Psi_{w_{i,j}}(\lambda^0_{i,j})\neq 0$. This finishes the proof of Theorem \ref{thm:sing}.

\subsection{Proof of Theorem \ref{thm:coh}} To prove Theorem 
\ref{thm:coh}, we first determine the Bernstein degree of 
$\overline X(\gamma_{\lambda, w_{i,j}})$ using Proposition \ref{prop:BD}.

\begin{prop} \label{cor:d}
Let $\lambda\in\Lambda$ be regular and dominant. For $i, j \in \{1, 2,\ldots, n\}$, $i\neq j$, the constant $d_{w_{i,j}}=1$ in \eqref{const}, that is,
\[
c(\overline X(\gamma_{\lambda,w_{i,j}}))=c_{w_{i,j}}(\lambda).
\]
\end{prop}

\begin{proof}

In view of  \cite[Section 5.12]{J80b}, we can define a linear map 
\[
 c: {\rm Coh}_\Lambda^{\min}(\CK(G))\rightarrow \BC[\frak h^*\times \frak h^*]
\]
such that for all $1\leq i,j\leq n$ and  all regular dominant elements $\lambda\in \Lambda$,
\[
c(\overline{\Psi}_{w_{i,j}})(\lambda)=\left\{
\begin{array}{ll}
  \textrm{the Bernstein degree of } \overline{\Psi}_{w_{i,j}}(\lambda),&\textrm{if $i\neq j$};\\
  0, &\textrm{if $i=j$}.
  \end{array}
  \right.
\]


By parabolic induction in stages,  $I_{\gamma_{\lambda,w_{i,j}}}(a_i, b_j)$ is a subquotient 
of $X(\gamma_{\lambda,w_{i,j}})$. The consideration of minimal $K$-types shows that 
$\overline X(\gamma_{\lambda,w_{i,j}})$ is a constituent of $I_{\gamma_{\lambda,w_{i,j}}}(a_i, b_j)$, hence it occurs with multiplicity one. By Lemma \ref{lem:BD} and Proposition \ref{prop:BD}, we have that
\[
c(I_{\gamma_{\lambda,w_{i,j}}}(a_i, b_j)) = \dim F_{i,j} = \sum_{(k,l)\in [i]\times [j]}c_{v_{k,l}}(\lambda).
\]
It is also straightforward to verify that $w_{i,j}=v_{k,l}$
for some $(k,l)\in [i]\times [j]$. 
The multiplicity one of the occurence of $\overline X(\gamma_{\lambda, w_{i,j}})$ in $I_{\gamma_{\lambda,w_{i,j}}}(a_i, b_j)$ implies that 
\[
c(\overline X(\gamma_{\lambda,w_{i,j}}))= c(\overline{\Psi}_{w_{i,j}})(\lambda) = c_{w_{i,j}}(\lambda).
\]
Here we use the fact that $c(\overline \Psi_{v_{i',j'}})\in \BC[\Fh^*\times \Fh^*]$,  $i', j' \in \{2,\ldots, n\}$, $i'\neq j'$, are linearly independent, which follows from the fact that \eqref{basis1} gives a basis of ${\rm Coh}^{\min}_\Lambda(\CK(G))$.
\end{proof}

Part (a) of Theorem \ref{thm:coh} is now a consequence of the following result. 

\begin{thm} \label{thm:composition}
Suppose that  $\lambda=(a_1,\ldots, a_n, b_1,\ldots, b_n)\in \Lambda$ is regular and dominant. With the notation as in \eqref{gammaw}, for all  $i,j\in\{1, 2, \ldots, n\}$, the equality 
\[
\left[ I_{\gamma_{\lambda,w_{i,j}}}(a_i, b_j) \right] = \sum_{(k,l)\in [i]\times[j]} \left[\overline X(\gamma_{\lambda,v_{k,l}})\right] + \delta_{i,j}\cdot \left[F_\lambda\right]
\]
holds in $\CK(G)$, where $F_\lambda$ is the irreducible finite-dimensional representation of $\mathrm{GL}_{n}(\BC)$ with infinitesimal character $\lambda$.
\end{thm}

\begin{proof} By Proposition \ref{prop:BD} and Proposition \ref{cor:d} (and its proof), $\overline X(\gamma_{\lambda,v_{k,l}})$, $(k,l)\in [i]\times[j]$, are all the constituents of $I_{\gamma_{\lambda,w_{i,j}}}(a_i, b_j)$ with Gelfand-Kirillov dimension $2n-2$, and each of them occurs with multiplicity one.

It remains to show that $F_\lambda$ occurs in $I_{\gamma_{\lambda,w_{i,j}}}(a_i, b_j)$ 
if and only if $i=j$, in which case it occurs with multiplicity one.
The case that $i=j$ can be proved similarly as in the  proof of Proposition \ref{cor:d}, which will be omitted.

Assume that $i\neq j$. Then we need to show that $F_\lambda$ does not occur in
$I_{\gamma_{\lambda,w_{i,j}}}(a_i, b_j)$. Since $I_{\gamma_{\lambda,w_{i,j}}}$ is a subquotient
of $X(\gamma_{\lambda,w_{i,j}})$, this follows from the general statement that $F_\lambda$ does not occur in  $X(\gamma_{\lambda,w})$ for any $w\neq 1$. The last statement should be known, but we provide a proof below for completeness.

Recall that $B$ is the upper triangular Borel subgroup of $G$, and $X(\gamma_{\lambda,w}) = \Ind^G_B(\chi_{\lambda,w})$, where
\[
\chi_{\lambda,w} := \chi_{a_1, b_{w^{-1}(1)}}\otimes\cdots \otimes \chi_{a_n, b_{w^{-1}(n)}}.
\]
It suffices to show that the minimal $K$-type of $F_\lambda$, which has extremal weight $\bd{a}-\bd{b}$ (see 
\eqref{ex-wt}), does not occur in $X(\gamma_{\lambda,w})$ for any $w\neq 1$. 
If we let $T\cong \RU(1)^n$ be the diagonal torus in $K=\RU(n)$, then 
\[
X(\gamma_{\lambda,w})\vert_K \cong \Ind^K_T (\chi_{\lambda,w}\vert_T).
\]
By Frobenius reciprocity, we need to show that $\bd{a}-w\bd{b}$, $w\neq 1$, is not a weight of the minimal
$K$-type of $F_\lambda$.

We recall an elementary fact (see \cite[Proposition 0.4.3 b)]{V81}): if $\xi, \zeta\in \BZ^n $ are  weights of an irreducible representation of $K$ and $\xi$ is extremal, then $\langle \xi, \xi\rangle \geq \langle \zeta, \zeta\rangle$, where $\langle \ ,\rangle$ is the standard inner product. 
Thus it is enough to prove that for $w\neq 1$,
\[
\langle \bd{a}-\bd{b}, \bd{a}- \bd{b}\rangle < \langle \bd{a}-w\bd{b}, \bd{a}-w\bd{b}\rangle.
\]
Without loss of generality, we assume that $\bd{a}, \bd{b}\in \BZ^n$. The above inequality is reduced to
\[
\langle \bd{a}, \bd{b}\rangle > \langle \bd{a}, w\bd{b}\rangle,
\]
which follows directly from the fact that $\lambda$ is regular and dominant and the well known rearrangement inequality.

This finishes the proof that $F_\lambda$ does not occur in $X(\gamma_{\lambda,w})$ for  $w\neq 1$,
and thereby finishes the proof of the theorem.
\end{proof}

It remains to prove parts  (b), (c) and (d) of Theorem \ref{thm:coh}, which are essentially consequences of part (a). Part (b) can be deduced from  (a) and the identity (see \eqref{hL1}) that for $i\in \{2,\ldots, n\}$,   
\[
p_{y_i}(- \bd{a}) = \sum_{i\leq k\leq n}(-1)^{k-i} h_L(-y_k \bd{a}).
\]
Part (c) follows from  (a) and \eqref{alt}. Part (d) follows from   (a) and  (b). The proof is finished.

\section*{Acknowledgement}

Z. Bai was supported in part by the National Natural Science Foundation of 
China (No. 12171344) and the National Key $\textrm{R}\,\&\,\textrm{D}$ Program of China (No. 2018YFA0701700 and No. 2018YFA0701701). Y. Chen was supported in part by the Natural Science Foundation of Jiangsu Province (No. BK20221057).
D. Liu was supported by National Key R \& D Program of China (No. 2022YFA1005300)  and the National Natural Science Foundation of China (No. 12171421). B. Sun was supported by  National Key R \& D Program of China (No. 2022YFA1005300 and 2020YFA0712600) and New Cornerstone Investigator Program.

\end{document}